\documentclass{scrartcl}

\usepackage[english]{babel}
\usepackage[T1]{fontenc}
\usepackage[french=guillemets]{csquotes}

\usepackage{hyperref}
\hypersetup{final,colorlinks=true,
linkcolor=[rgb]{0.1, 0.1, 0.44},
citecolor=[rgb]{0.1, 0.1, 0.44}}

\usepackage{amsmath}
\usepackage{amsfonts}
\usepackage{amssymb,mathabx}

\usepackage{amsthm} 

\theoremstyle{plain} 
\newtheorem{theorem}{Theorem} 
\newtheorem{scholium} [theorem]{Scholium}
\newtheorem{lemma} [theorem]{Lemma}

\theoremstyle{remark}
\newtheorem*{question} {Question}

\newcommand{\G}{\mathcal{G}}
\newcommand{\D}{\mathcal{D}}

\newcommand{\SHIFT}{\mathsf{S}}

\newcommand{\bcn}{\chi_{B}} 

\newcommand{\SSI}{\quad\longleftrightarrow\quad}
\newcommand{\ssi}{\ \leftrightarrow \ }

\newcommand{\infSub}[1]{[#1]^{\infty}}
\newcommand{\finSub}[1]{[#1]^{<\infty}}

\newcommand{\unf}[1]{\vec{#1}}

\newcommand{\segm}{\sqsubseteq}

\newcommand{\Eff}{\mathcal{F}(\infSub{\omega})}
\newcommand{\Tr}{\mathcal{T}}
\newcommand{\Refl}{\mathcal{R}}

\newcommand{\bai}{\omega^\omega}

\begin{document}

\title{Finite versus infinite: an insufficient shift}

\author{Yann Pequignot \thanks{The author gratefully acknowledges the support of Austrian Science Fund (FWF) through project \texttt{I~1238} and the support of the Swiss National Science Foundation (SNF) through grant \texttt{P2LAP2$\_$164904}.}}


%

\maketitle

\begin{abstract}
The shift graph $\G_{\SHIFT}$ is defined on the space of infinite subsets of natural numbers by letting two sets be adjacent if one can be obtained from the other by removing its least element. 
We show that this graph is not a minimum among the graphs of the form $\G_{f}$ defined on some Polish space $X$, where two distinct points are adjacent if one can be obtained from the other by a given Borel function $f:X\to X$. This answers the primary outstanding question from \cite{Kechris19991}.
\end{abstract}

%

A \emph{directed graph} is a pair $\G=(X,R)$ where $R$ is an irreflexive binary relation on $X$. A homomorphism from $\G=(X,R)$ to $\G'=(X',R')$ is a map $h:X\to X'$ such that $(x,y)\in R$ implies $(h(x),h(y))\in R'$ for all $x,y\in X$. A \emph{coloring} of $\G$ is a map $c:X\to Y$ such that $(x_{1},x_{2})\in R$ implies $c(x_{1})\neq c(x_{2})$ for all $(x_{1},x_{2})\in X\times X$. In case $X$ is a topological space, the \emph{Borel chromatic number} $\bcn(\G)$ of $\G$ is defined by
\[
\bcn(\G)=\min \{ |c(X)| \mid \text{$c:X\to Y$ is a Borel coloring of $\G$ in a Polish space $Y$}\},
\]
where $|c(X)|$ denotes the cardinality of the range of $c$.

In this note we only deal with graphs generated by a function. Let $X$ be a Polish space and $f:X\to X$ is a Borel map. We let $\D_{f}=(X,D_{f})$ be the directed graph given by
\[
x\mathrel{D_{f}}y \SSI x\neq y \land f(x)=y.
\]
We also consider its symmetric counterpart $\G_{f}=(X,R_{f})$ given by
\[
x\mathrel{R_{f}}y \SSI x\neq y \land (f(x)=y \lor f(y)=x).
\]
Notice that clearly $\bcn(\G_{f})=\bcn(\D_{f})$.

The following example has drawn considerable attention in the study of Borel chromatic numbers \cite{conley2014antibasis,di2006canonical,di2012shift,prisco2015basis}. Let $X=\infSub{\omega}$ be the set of infinite sets of natural numbers with topology induced from the Cantor space $2^{\omega}$ when $Y\subseteq \omega$ is identified with its characteristic function $\chi_{Y}:\omega\to 2$. The shift operation $\SHIFT:\infSub{\omega}\to \infSub{\omega}$ is the continuous map defined by $\SHIFT(Y)=Y\setminus \{\min Y\}$. While $\G_{\SHIFT}$ is an acyclic graph and has therefore chromatic number $2$, it follows from the Galvin--Prikry Theorem \cite{JSL:9194679} that $\bcn(\G_{\SHIFT})=\aleph_{0}$. 

Kechris, Solecki and Todor{\v{c}}evi{\'c} \cite[Problem 8.1]{Kechris19991} (see also \cite[Section 3]{prisco2015basis} and \cite{miller2008measurable}) asked whether the following is true: If $X$ is a Polish space and $f:X\to X$ is a Borel function, then exactly one of the following holds:
\begin{enumerate}
\item The Borel chromatic number of $\G_{f}$ is finite;
\item There is a continuous homomorphism from $\G_{\SHIFT}$ to $\G_{f}$.
\end{enumerate}

We show that the answer is negative, namely:

\begin{theorem}\label{mainThm}
There exists a Polish space $X$ together with a continuous finite-to-$1$ function $f:X\to X$ such that $\bcn(\G_{f})=\aleph_{0}$ and there is no Borel homomorphism from $\G_{\SHIFT}$ to $\G_{f}$.
\end{theorem}

We do not have any explicit example witnessing the above existential statement. This is because our proof consists of showing that a certain subset of the set of graphs with the above property is a true $\mathbf{\Pi}^1_2$ set in some suitable standard Borel space. 

We can however be a bit more specific. If $P$ is a binary relation on $\omega$, let $\unf{P}$ be the closed subset of $\infSub{\omega}$ defined by 
\[
\unf{P}=\bigl\{(n_{i})_{i\in\omega}\in\infSub{\omega}\mid \forall i\in\omega \ n_{i}\mathrel{P}n_{i+1} \bigr\},
\]
where an element of $\infSub{\omega}$ is identified with the enumeration $(n_{i})_{i\in\omega}$ of its elements in strictly increasing order. If $\G=(X,R)$ is a directed graph and $Y\subseteq X$ let us denote by $\G|Y$ the restriction of $\G$ to $Y$ given by $(Y, R\cap (Y\times Y))$.

The proof of Theorem~\ref{mainThm} actually yields the following result.

\begin{scholium} \label{scholium}
There exists a binary relation $P$ on $\omega$ such that $\bcn(\G_{\SHIFT}|\unf{P})=\aleph_{0}$ and there is no Borel homomorphism from $\G_{\SHIFT}$ to $\G_{\SHIFT}|\unf{P}$.
\end{scholium}

Before proving Theorem~\ref{mainThm} we want to recall a definition and establish a simple but important lemma. A binary relation $P\subseteq A\times A$ on some set $A$ is called a \emph{better-binary-relation} if for all continuous map $\varphi:\infSub{\omega}\to A$, where $A$ is considered a discrete space, there exists $X\in \infSub{\omega}$ such that $\bigl(\varphi(X),\varphi(\SHIFT(X) )\bigr)\in P$, or in fewer words, if there is no continuous homomorphism from $\D_{\SHIFT}$ to $(A,P^{\complement})$, where $P^{\complement}=(A\times A)\setminus P$. This notion which first appeared in \cite{shelah1982better} is a straightforward generalization to arbitrary binary relations of that of \emph{better-quasi-order}\footnote{A \emph{better-quasi-order} is just a transitive better-binary-relation, as a better-binary-relation is necessarily reflexive.} due to Nash-Williams \cite{nash1965well}. For more on better-quasi-orders we refer the reader to \cite{simpson1985bqo, marcone1994foundations, CarroyYPFromWell} and to the author's PhD thesis \cite{PhDThesisYP}.

\begin{lemma}\label{lemDiscrete}
Let $P\subseteq \omega\times \omega$ be an irreflexive binary relation on $\omega$. Then the following are equivalent:
\begin{enumerate}
\item $P^{\complement}$ is not a better-binary-relation, \label{lemDiscrete1}
\item there exists a continuous homomorphism from $\D_{\SHIFT}$ to $(\omega,P)$, \label{lemDiscrete2}
\item there exists a continuous homomorphism from $\D_{\SHIFT}$ to $\D_{\SHIFT}|\unf{P}$. \label{lemDiscrete3}
\item there exists a Borel homomorphism from $\G_{\SHIFT}$ to $\G_{\SHIFT}|\unf{P}$. \label{lemDiscrete4}
\end{enumerate} 
\end{lemma}
\begin{proof}
(\ref{lemDiscrete1}) $\ssi$ (\ref{lemDiscrete2}) follows from the definition of a better-binary-relation. 

(\ref{lemDiscrete2}) $\to$ (\ref{lemDiscrete3}): Assume that $\varphi: \infSub{\omega}\to \omega$ is a continuous homomorphism from $\D_{\SHIFT}$ to $(\omega,P)$. Since the usual order on $\omega$ is a better-quasi-order, by applying the Galvin--Prikry theorem to the Borel partition:
\[
\infSub{\omega}=\bigl\{X \bigm\vert \Phi(X)\leq\Phi(\SHIFT(X))\bigr\}\cup \bigl\{X\bigm \vert \Phi(X)>\Phi(\SHIFT(X))\bigr\},
\]
and eventually restricting $\varphi$ to $\infSub{Y}$ for some $Y\in\infSub{\omega}$, we can assume without loss of generality that $\varphi(X)\leq \varphi(\SHIFT(X))$ for every $X\in\infSub{\omega}$. As $P$ is irreflexive and $\varphi$ is a homomorphism, we actually have $\varphi(X)<\varphi(\SHIFT(X))$ for every $X\in\infSub{\omega}$. We define $\Phi:\infSub{\omega}\to \unf{P}$ by setting $\Phi(X)=\{\varphi(\SHIFT^{n}(X))\mid n\in\omega\}$ for every $X\in\infSub{\omega}$. Clearly $\Phi$ is a well defined continuous homomorphism from $\D_{\SHIFT}$ to $\D_{\SHIFT}|\unf{P}$ as desired. 

(\ref{lemDiscrete3}) $\to$ (\ref{lemDiscrete4}) is obvious.

(\ref{lemDiscrete4}) $\to$ (\ref{lemDiscrete2}): Suppose that $\Phi$ is a Borel homomorphism from $\G_{\SHIFT}$ to $\G_{\SHIFT}|\unf{P}$. Applying the Galvin--Prikry theorem to the Borel partition
\[
\infSub{\omega}=\bigl\{X \bigm\vert \SHIFT(\Phi(X))=\Phi(\SHIFT(X))\bigr\}\cup \bigl\{X\bigm \vert \SHIFT(\Phi(\SHIFT(X)))=\Phi(X)\bigr\},
\]
and eventually restricting $\Phi$ to $\infSub{Y}$ for some $Y\in\infSub{\omega}$, we can suppose without loss of generality that $\Phi$ is actually a homomorphism from $\D_{\SHIFT}$ to $\D_{\SHIFT}|\unf{P}$. By eventually restricting further $\Phi$ to $\infSub{Z}$ for some $Z\in\infSub{\omega}$, we can assume that $\Phi$ is continuous (\cite[section 6]{mathias1977happy}, \cite[Theorem 3.5]{simpson1985bqo} and \cite[Proposition 3.2]{promel1992wqo}). We then define $\varphi:\infSub{\omega}\to \omega$ by $\varphi(X)=\min \Phi(X)$ for all $X\in\infSub{\omega}$. Since $\Phi$ is a homomorphism from $\D_{\SHIFT}$ to $\D_{\SHIFT}$, we have $\varphi(\SHIFT(X))=\min \Phi(\SHIFT(X))=\min \SHIFT(\Phi(X))$, and as $\Phi(X)\in\unf{P}$ it follows that $\varphi(X)\mathrel{P} \varphi(\SHIFT(X))$ for all $X\in\infSub{\omega}$. Hence $\varphi$ is a homomorphism from $\D_{\SHIFT}$ to $(\omega, P)$ and clearly $\varphi$ is continuous, as desired. 
\end{proof}

\begin{proof}[Proof of Theorem~\ref{mainThm}]
We confine ourselves to the graphs $\G_{\SHIFT}|C$ obtained by restricting $\G_{\SHIFT}$ to some closed subset $C$ of $\infSub{\omega}$ closed under the shift operation, i.e. such that $\SHIFT(X)\in C$ for all $X\in C$. First some notation. Let $\finSub{\omega}$ be the set of finite sets of natural numbers. For $s\in \finSub{\omega}$ and $t\subseteq \omega$ let $s\segm t$ denote that $s$ is an \emph{initial segment} of $t$ with respect to the usual order on $\omega$, namely $s\segm t$ if and only if $s=t$ or $\exists k\in t$ such that $s=\{n\in t\mid n<k\}$.

We consider the Effros Borel space $\Eff$ of closed subsets of $\infSub{\omega}$ (see \cite[12.C]{kechris1995classical}). The $\sigma$-algebra of Borel sets of $\Eff$ is generated by the sets of the form 
\[
\{F\in \Eff \mid F\cap [s] \neq \emptyset\}
\]
where $[s]=\{X\in\infSub{\omega}\mid s\segm X\}$, for $s\in \finSub{\omega}$. We identify every closed subset $F$ of $\infSub{\omega}$ with the pruned tree $T_{F}=\{t\in\finSub{\omega}\mid \exists X\in F \ t\segm X\}$ that we view as an element of the product space $2^{\finSub{\omega}}$, with $\finSub{\omega}$ discrete. Notice that the Effros Borel structure on $\Eff$ coincides with the Borel structure induced by $2^{\finSub{\omega}}$ via this identification. Let $\Tr$ be the set of closed sets $F\in\Eff$ such that $\SHIFT(X)\in F$ for all $X\in F$. Clearly $\Tr$ is Borel in $\Eff$ since
\[
F\in \Tr \SSI \forall t \in \finSub{\omega} \ \forall n [(n <\min t \land \{n\}\cup t\in T_{F}) \to   t\in T_{F}].
\]
We henceforth work within the standard Borel space $\Tr$ (\cite[(13.4)]{kechris1995classical}). The two subsets of $\Tr$ that we are interested in are the following: 
\begin{align*}
\mathcal{N}=&\big\{F\in \Tr \mid \text{there is no Borel homomorphism from $\G_{\SHIFT}$ to $\G_{\SHIFT}|F$} \big\},\\
\mathcal{F}=&\big\{F\in \Tr \mid \bcn(\G_{\SHIFT}|F)<\aleph_{0} \big\}.
\end{align*}

Clearly $\mathcal{F}\subseteq \mathcal{N}$, as the composition of a coloring with a homomorphism is again a coloring and $\bcn(\G_\SHIFT)=\aleph_0$. First we observe that $\mathcal{F}$ is a $\mathbf{\Sigma}^{1}_{2}$ set in $\Tr$. While this can be seen by a direct Tarski--Kuratowski computation on the definition of $\mathcal{F}$, we find it easier to use the neat characterization given by Miller \cite[Thm 2.1]{miller2008measurable} which gives: $F\in \mathcal{F}$ if and only if there exists a Borel set $B\subseteq F$ such that for all $X\in F$ there exist $m,n\in\omega$ such that $\SHIFT^m(X)\in B$ and $\SHIFT^n(X)\not\in B$. We fix a coding of Borel subsets of the Polish space $\infSub{\omega}$, see \cite[(35.B)]{kechris1995classical}: let $D$ be a $\mathbf{\Pi}^1_1$ subset of $\omega^\omega$ and $W$ be a $\mathbf{\Delta}^1_1$ subset of $D\times \infSub{\omega}$ such that $\{W_d \mid d\in D\}$, where $W_d=\{X\mid (d,X)\in W\}$, is the set of all Borel subsets $\infSub{\omega}$. 
We get
\begin{align*}
F\in\mathcal{F} \SSI & \exists d\in \bai \bigg( d\in D\\
&\land \forall X \in\infSub{\omega} \Big[X\in F \to \big( \exists n \ \SHIFT^{n}(X)\in W_{d} \land \exists n \ \SHIFT^{n}(X)\notin W_{d}\big)\Big]\bigg).
\end{align*}
This clearly gives a $\mathbf{\Sigma}^{1}_{2}$ definition of $\mathcal{F}$ in $\Tr$. 

Next we show that the inclusion $\mathcal{F}\subsetneq \mathcal{N}$ is strict as witnessed by a closed set of the form $\unf{P^\complement}$ for some reflexive binary relation $P$ on $\omega$. This will prove Scholium~\ref{scholium} and finish the proof.

To achieve this, we rely on a deep result due to Marcone \cite{marcone1994foundations,marcone1995complete} that we now recall. The set $\mathcal{B}=\{Q\subseteq \omega\times\omega \mid \text{$Q$ is a better-binary-relation} \}$ is a $\mathbf{\Pi}^{1}_{2}$-complete subset of the compact Polish space $\Refl=\{P\subseteq \omega\times \omega \mid \text{$P$ is reflexive}\}$, where $\Refl$ is endowed with the topology induced by the product space $2^{\omega\times\omega}$, so in particular $\mathcal{B}$ is not $\mathbf{\Sigma}^{1}_{2}$ in $\Refl$. For every reflexive $P\subseteq \omega\times\omega$ let 
\[
f(P)=\bigl\{(n_{i})_{i\in\omega}\in\infSub{\omega}\bigm\vert \forall i \ \neg n_{i} \mathrel{P} n_{i+1} \bigr\}=\unf{P^{\complement}}.
\]
This defines a $\mathbf{\Delta}^{1}_{2}$-measurable reduction $f:\Refl \to \Tr$ from $\mathcal{B}$ to $\mathcal{N}$. To see that $f^{-1}(\mathcal{N})=\mathcal{B}$, let $P\subseteq \omega\times \omega$ be a reflexive relation. Then $P^{\complement}$ is an irreflexive relation and by Lemma~\ref{lemDiscrete}, $P$ is a better-binary-relation if and only if $f(P)=\unf{P^{\complement}}\in\mathcal{N}$.

To see that $f$ is $\mathbf{\Delta}^{1}_{2}$-measurable notice that for every basic open set $[s]=\{X\in\infSub{\omega}\mid s\segm X\}$ of $\infSub{\omega}$ we have 
\[
f(P)\cap [s]\neq\emptyset \SSI \exists (n_{i})_{i\in\omega}\in\infSub{\omega} \bigl(\forall i \ \neg n_{i}\mathrel{P} n_{i+1} \land s\segm (n_{i})_{i\in\omega}\bigr).
\]
It follows that for every Borel subset $C$ of $\Tr$ the set $f^{-1}(C)$ belongs to the $\sigma$-algebra generated by $\mathbf{\Sigma}^{1}_{1}$, and \textit{a fortiori} to $\mathbf{\Delta}^{1}_{2}$. 

Suppose towards a contradiction that $f(P)\in \mathcal{N}$ implies $f(P)\in \mathcal{F}$ for all $P\in \Refl$. Then it follows that $f^{-1}(\mathcal{F})=\mathcal{B}$. But since $\mathcal{F}$ is $\mathbf{\Sigma}^{1}_{2}$ and $\mathbf{\Sigma}^{1}_{2}$ is closed under preimages by $\mathbf{\Delta}^{1}_{2}$-measurable functions \cite[(37.3)]{kechris1995classical}, we get that $\mathcal{B}$ is $\mathbf{\Sigma}^{1}_{2}$ contradicting Marcone's Theorem. Therefore there exists some $P\in \Refl$ such that $f(P)\in \mathcal{N}\setminus \mathcal{F}$ as desired.
 \end{proof}

It would be very interesting to find an explicit example of a graph $\G_f$ whose existence is guaranteed by Theorem~\ref{mainThm}. Notice that by a direct application of a result due to Pouzet \cite[Theorem 7]{pouzet1993graphs} (see also \cite[Theorem 1.8]{marcone1994foundations}), the binary relation $P$ in Scholium~\ref{scholium} can be chosen such that $P^\complement=(\omega\times \omega) \setminus P$ is actually a quasi-order, and therefore a better-quasi-order by Lemma~\ref{lemDiscrete}. We end this note by giving an example of a graph $\G_f$ generated in a similar fashion from a countable better-quasi-order and for which we do not know the Borel chromatic number.

Consider the set $2^{<\omega}$ of finite binary words equipped with the subword ordering, i.e.
\begin{align*}
u\preccurlyeq v \SSI &\text{there exists a strictly increasing map $h:|u|\to |v|$}\\
&\text{such that for every $i<|u|$ we have $u(i)=v(h(i))$}, 
\end{align*}
where $|u|$ denotes the length of $u\in 2^{<\omega}$. Consider the closed subspace 
\[
X=\bigl\{(u_n)_{n\in\omega}\in (2^{<\omega})^\omega\bigm\vert \forall n \ u_n \npreccurlyeq u_{n+1}\bigr\}
\]
of the product space $(2^{<\omega})^\omega$, where $2^{<\omega}$ is discrete, and define the corresponding shift operation $\SHIFT': X\to X$, $(u_n)_{n\in\omega}\mapsto (u_{n+1})_{n\in\omega}$. Since $(2^{<\omega}, \preccurlyeq)$ is a better-quasi-order,
 there is no Borel homomorphism from $\G_\SHIFT$ to $\G_{\SHIFT'}$ as the proof of Lemma~\ref{lemDiscrete} also shows. We however ask the following:
 
\begin{question}  
What is the Borel chromatic number of $\G_{\SHIFT'}$?
\end{question}

 \section*{Acknowledgements}
I would like to thank Benjamin D. Miller for numerous pleasant and stimulating discussions in Vienna. I also wish to thank Rapha\"el Carroy and Gianluca Basso for making useful suggestions and pointing out several misprints.


\begin{thebibliography}{DPT15}

\bibitem[CM14]{conley2014antibasis}
Clinton Conley and Benjamin Miller.
\newblock An antibasis result for graphs of infinite borel chromatic number.
\newblock {\em Proceedings of the American Mathematical Society},
  142(6):2123--2133, 2014.

\bibitem[CP14]{CarroyYPFromWell}
Rapha\"{e}l Carroy and Yann Pequignot.
\newblock From well to better, the space of ideals.
\newblock {\em Fundamenta Mathematicae}, 227(3):247--270, 2014.

\bibitem[DPT06]{di2006canonical}
Carlos~A Di~Prisco and Stevo Todor{\v{c}}evi{\'c}.
\newblock Canonical forms of shift-invariant maps on $[\mathbb{N}]^{\omega}$.
\newblock {\em Discrete mathematics}, 306(16):1862--1870, 2006.

\bibitem[DPT12]{di2012shift}
Carlos~A Di~Prisco and Stevo Todor{\v{c}}evi{\'c}.
\newblock Shift graphs on precompact families of finite sets of natural
  numbers.
\newblock {\em Discrete Mathematics}, 312(19):2915--2926, 2012.

\bibitem[DPT15]{prisco2015basis}
Carlos~A Di~Prisco and Stevo Todor{\v{c}}evi{\'c}.
\newblock Basis problems for {B}orel graphs.
\newblock {\em Zbornik Radova}, 17(25):33--51, 2015.

\bibitem[GP73]{JSL:9194679}
Fred Galvin and Karel Prikry.
\newblock {B}orel sets and {R}amsey's theorem.
\newblock {\em The Journal of Symbolic Logic}, 38:193--198, 6 1973.

\bibitem[Kec95]{kechris1995classical}
Alexander~S Kechris.
\newblock {\em Classical descriptive set theory}, volume 156.
\newblock Springer-Verlag New York, 1995.

\bibitem[KST99]{Kechris19991}
Alexander~S. Kechris, S{\l}awomir Solecki, and Stevo Todor{\v{c}}evi{\'c}.
\newblock {B}orel chromatic numbers.
\newblock {\em Advances in Mathematics}, 141(1):1 -- 44, 1999.

\bibitem[Mar94]{marcone1994foundations}
Alberto Marcone.
\newblock Foundations of bqo theory.
\newblock {\em Transactions of the American Mathematical Society},
  345(2):641--660, 1994.

\bibitem[Mar95]{marcone1995complete}
Alberto Marcone.
\newblock The set of better quasi orderings is $\mathbf{\Pi}^1_2$-complete.
\newblock {\em Mathematical Logic Quarterly}, 41:373--383, 1995.

\bibitem[Mat77]{mathias1977happy}
Adrian~RD Mathias.
\newblock Happy families.
\newblock {\em Annals of Mathematical Logic}, 12(1):59--111, 1977.

\bibitem[Mil08]{miller2008measurable}
Benjamin~D Miller.
\newblock Measurable chromatic numbers.
\newblock {\em The Journal of Symbolic Logic}, 73(04):1139--1157, 2008.

\bibitem[NW65]{nash1965well}
Crispin St. John~Alvah Nash-Williams.
\newblock On well-quasi-ordering transfinite sequences.
\newblock In {\em Proc. Cambridge Philos. Soc}, volume~61, pages 33--39.
  Cambridge University Press, 1965.

\bibitem[Peq15]{PhDThesisYP}
Yann Pequignot.
\newblock {\em Better-quasi-order: ideals and spaces}.
\newblock Phd thesis, Universit{\'e} de Lausanne and Universit{\'e} Paris
  Diderot -- Paris 7, 2015.

\bibitem[Pou93]{pouzet1993graphs}
Maurice Pouzet.
\newblock Graphs and posets with no infinite independent set.
\newblock In {\em Finite and Infinite Combinatorics in Sets and Logic}, pages
  313--335. Springer, 1993.

\bibitem[PV92]{promel1992wqo}
Hans~J{\"u}rgen Pr{\"o}mel and Bernd Voigt.
\newblock From wqo to bqo, via {E}llentuck's theorem.
\newblock {\em Discrete mathematics}, 108(1-3):83--106, 1992.

\bibitem[She82]{shelah1982better}
Saharon Shelah.
\newblock Better quasi-orders for uncountable cardinals.
\newblock {\em Israel Journal of Mathematics}, 42(3):177--226, 1982.

\bibitem[Sim85]{simpson1985bqo}
Stephen~G. Simpson.
\newblock Bqo theory and {F}ra{\"\i}ss{\'e}'s conjecture.
\newblock In Richard Mansfield and Galen Weitkamp, editors, {\em Recursive
  aspects of descriptive set theory}, pages 124--138. Oxford University Press,
  1985.

\end{thebibliography}
\end{document}